\documentclass[12pt]{amsart}

\usepackage{color,fancyhdr}
\usepackage[all]{xy}

\usepackage{amsmath,amsthm,mathrsfs,amsfonts,
amssymb,times,latexsym,mathabx}

\usepackage[usenames,dvipsnames,svgnames,table]{xcolor}

\DeclareMathAlphabet{\curly}{U}{rsfs}{m}{n}  

\newtheorem{theorem}{Theorem}[section]
\newtheorem{lemma}{Lemma}[section]

\theoremstyle{definition}

\newtheorem{remark}[theorem]{Remark}
\theoremstyle{problem}
\newtheorem{problem}[theorem]{Problem}
\usepackage{xcolor}

\pagecolor[rgb]{0.9,0.99,0.9}

\numberwithin{equation}{section}

\makeatletter
\renewcommand{\pmod}[1]{\allowbreak\mkern7mu({\operator@font mod}\,\,#1)}
\makeatother

\newcommand{\be}{\begin{equation}}
\newcommand{\ee}{\end{equation}}

\renewcommand{\le}{\leqslant}

\renewcommand{\ge}{\geqslant}

\begin{document}

\title[On sets with sum and difference structure]
{On sets with sum and difference structure}

\author{Jin-Hui Fang}
\address{Department of Mathematics, Nanjing University of Information Science $\&$ Technology, Nanjing 210044, PR China}
\email{fangjinhui1114@163.com}
\author{Csaba S\'{a}ndor*}
\address{Department of Stochastics, Institute of Mathematics, Budapest University of Technology and Economics, M\H{u}egyetem rkp. 3., H-1111, Budapest, Hungary; Department of Computer Science and Information Theory, Budapest University of Technology and Economics, M\H{u}egyetem rkp. 3., H-1111 Budapest, Hungary; MTA-BME Lend\"{u}let Arithmetic Combinatorics Research Group, ELKH, M\H{u}egyetem rkp. 3., H-1111 Budapest, Hungary}
\email{csandor@math.bme.hu}
\thanks{* Corresponding author.}
\thanks{The first author is supported by the National Natural Science Foundation of China, Grant No. 12171246 and the Natural Science Foundation of Jiangsu Province, Grant No. BK20211282. The second author is supported by the NKFIH Grants No. K129335.}
\date{\today}%

\begin{abstract}
For nonempty sets $A,B$ of nonnegative integers and an integer $n$, let $r_{A,B}(n)$ be the number of representations of $n$ as $a+b$ and $d_{A,B}(n)$ be the number of representations of $n$ as $a-b$, where $a\in A, b\in B$. In this paper, we determine the sets $A,B$ such that $r_{A,B}(n)=1$ for every nonnegative integer $n$. We also consider the \emph{difference} structure and prove that: there exist sets $A$ and $B$ of nonnegative integers such that
$r_{A,B}(n)\ge 1$ for all large $n$, $A(x)B(x)=(1+o(1))x$ and for any given nonnegative integer $c$, we have $d_{A,B}(n)=c$ for infinitely many positive integers $n$. Other related results are also contained.
\end{abstract}

\maketitle


\section{\bf Introduction}

Let $\mathbb{N}_0$ be the set of non-negative integers. For nonempty sets $A,B\subseteq \mathbb{N}_0$ and an integer $n$, let $r_{A,B}(n)$ be the number of representations of $n$ as $a+b$ and $d_{A,B}(n)$ be the number of representations of $n$ as $a-b$, where $a\in A, b\in B$. Two infinite sequences $A$ and $B$ are called \emph{infinite additive complements}, if
their sum $$A+B=\{a+b:a\in A, b\in B\}$$ contains all nonnegative integers, namely $r_{A,B}(n)\ge 1$ for all nonnegative integers. Let $A(x)$(resp. $B(x)$) be the number of elements in $A$(resp. $B$) not exceeding $x$.\vskip2mm

It follows from the definition that, for any infinite additive complements $A,B$, we have
$$\liminf_{x\rightarrow\infty}\frac{A(x)B(x)}{x}\ge 1.$$ In 1964, Danzer \cite{Danzer} disproved a conjecture posed by Hanani, that is:\vskip2mm

\noindent{\textbf{Theorem A (Danzer).}} There exist infinite additive complements $A$ and $B$ such that \begin{equation}\label{1.1b}\lim_{x\to +\infty}\frac{A(x)B(x)}{x}=1.\end{equation}

Recently, Kiss and S\'andor \cite{Sandor} extended the above result. In fact, they proved the following nice result:\vskip2mm

\noindent{\textbf{Theorem B (\cite[Theorem 2]{Sandor}).}} For each integer $h\ge 2$ there exist infinite sets of nonnegative integers $A_1,\cdots,A_h$ with the following properties:\vskip2mm

\noindent(1) $A_1+\cdots+A_h=\mathbb{N}_0$,\vskip1mm

\noindent(2) $A_1(x)\cdots A_h(x)=(1+o(1))x$ as $x\rightarrow\infty$.\vskip2mm

In this paper, we determine the sets $A,B$ such that $r_{A,B}(n)=1$ for every nonnegative integer $n$. We also consider the \emph{difference} structure.
\vskip2mm

\begin{theorem}\label{thm1} $r_{A,B}(n)=1$ for every nonnegative integer $n$, if and only if
\begin{eqnarray}\label{05101}
&&A=\{\epsilon_0+\epsilon_2m_1m_2+\cdots+\epsilon_{2k-2}m_1\cdots m_{2k-2}+\cdots, \epsilon_{2i}=0,1,\cdots,m_{2i+1}-1\},\nonumber\\
&&B=\{\epsilon_1m_1+\epsilon_3m_1m_2m_3+\cdots+\epsilon_{2k-1}m_1\cdots m_{2k-1}+\cdots, \epsilon_{2i-1}=0,1,\cdots,m_{2i}-1\},
\end{eqnarray}(or $A,B$ interchanged), where $m_1,m_2,\cdots$ are integers no less than two.
\end{theorem}

Based on Theorem \ref{thm1}, we naturally posed the following problem for further research:
\begin{problem}\label{05121}
Is it true that if $A$ and $B$ are infinite additive complements and not of the form \eqref{05101}, then $r_{A,B}(n)\ge 2$ for infinitely many positive integer $n$?
\end{problem}

\begin{theorem}\label{thm2} If $r_{A,B}(n)=1$ for every nonnegative integer $n$, then $d_{A,B}(n)=1$ for every integer $n$.
\end{theorem}

We could deduce from Theorem \ref{thm2} that, if $A$ and $B$ are infinite additive complements and $d_{A,B}(n)\le 1$ for every integer $n$, then $r_{A,B}(n)=1$ for every nonnegative integer $n$, otherwise $r_{A,B}(n)\ge 2$ for some $n$, that is $n=a+b=a'+b'$. Hence $a-b'=a'-b$, which is impossible. We posed another problem similar to Problem \ref{05121}:

\begin{problem}\label{2} Is it true that if $A$ and $B$ are infinite additive complements and not of the form (1.2), then $d_{A,B}(n)\ge 2$ for infinitely many integer $n$?\end{problem}

\begin{theorem}\label{thm3} If $r_{A,B}(n)=1$ for every nonnegative integer $n$, then \begin{eqnarray*}\liminf_{x\rightarrow\infty}\frac{A(x)B(x)}{x}=1\hskip3mm \mbox{and}\hskip3mm \frac{3}{2}\le \limsup_{x\rightarrow\infty}\frac{A(x)B(x)}{x}\le2.\end{eqnarray*}
Furthermore, the constants $\frac{3}{2}$ and $2$ are best possible.
\end{theorem}

\begin{remark} By Theorem \ref{thm3}, we know that there does not exist infinite additive complements $A,B$ such that $r_{A,B}(n)=1$ for every nonnegative integer $n$ and \eqref{1.1b} holds.
\end{remark}

\begin{theorem}\label{thm4}
There exist infinite additive complements $A$ and $B$ of nonnegative integers such that \begin{equation}\label{1}\lim_{x\to +\infty}\frac{A(x)B(x)}{x}=1\end{equation}
and for any given nonnegative integer $c$, we have $d_{A,B}(n)=c$ for infinitely many positive integers $n$.\end{theorem}

\section{\bf Proof of Main Results}

\noindent\textbf{Proof of Theorem \ref{thm1}.} \emph{Sufficiency.} We only need to consider sets $A,B$ with the form \eqref{05101}, the other case is completely similar. Noting the fact that every nonnegative integer $n$ can be uniquely written in the form $\varepsilon_0+\varepsilon_1m_1+\varepsilon_2m_1m_2+\varepsilon_3m_1m_2m_3+\cdots$, where  $\varepsilon_i=0,1,\cdots,m_{i+1}-1$, we have $r_{A,B}(n)=1$ for every nonnegative integer $n$.\vskip2mm

Before the proof of \emph{Necessity}, we first prove the following preliminary lemma.\vskip2mm

\begin{lemma}\label{lemma1} Suppose that $C,D\subseteq\mathbb{N}$, $1\in C$ and $r_{C,D}(n)=1$ for every nonnegative integer $n$, then there exist an integer $m\ge 2$ and sets $E,F\subset \mathbb{N}$, $0\in E$, $0,1\in F$, $r_{E,F}(n)=1$ for every nonnegative integer $n$ such that $C=\{0,1,\cdots,m-1\}+mE$ and $D=mF$.\end{lemma}

\begin{proof} By $r_{C,D}(0)=1$ we know that $0\in C\cap D$. It follows from $r_{C,D}(n)=1$ for every nonnegative integer $n$ that $C\bigcap D=\{0\}$. Since $1\in C$, there exists an integer $m\ge 2$ such that $0,1,\cdots,m-1\in C$ and $m\not\in C$. Then, $1,...,m-1\not\in D$. It follows from $r_{C,D}(m)=1$ and $0\in C$ that $m\in D$. Now we will prove that there exist an integer $m\ge 2$ and sets $E,F\subseteq\mathbb{N}$, $0\in E$, $0,1\in F$, $r_{E,F}(n)=1$ for every nonnegative integer $n$ such that $C=\{0,1,\cdots,m-1\}+mE$ and $D=mF$. Assume the contrary. Suppose that $n$ is the least positive integer which destroys the above property. Then $n$ is not divisible by $m$. Noting that $C\bigcap D=\{0\}$, we divide into the following three cases:\vskip1mm

\noindent{\emph{Case 1}}. $n\in C$, $n\notin D$. Then $\lfloor \frac{n}{m}\rfloor m$, $\lfloor \frac{n}{m}\rfloor m+1,\cdots,n-1\not\in C$. It follows from $r_{C,D}(\lfloor \frac{n}{m}\rfloor m)=1$ that there exists an integer $d$ with $0<d<\lfloor\frac{n}{m}\rfloor m$, $m|d$ and $d\in D$ such that $\lfloor\frac{n}{m}\rfloor m-d\in C$. Since
\begin{eqnarray*}m|\lfloor\frac{n}{m}\rfloor m-d\hskip3mm \mbox{and}\hskip3mm \lfloor \frac{n}{m}\rfloor m-d<\lfloor\frac{n}{m}\rfloor m,\end{eqnarray*} so $n-d\in C$. Hence $n=(n-d)+d=n+0$, where $n-d,n\in C$ and $0,d\in D$, then $r_{C,D}(n)\ge 2$, a contradiction.\vskip1mm

\noindent{\emph{Case 2}}. $n\not\in C, n\in D$.\vskip1mm

Subcase 2.1. $\lfloor\frac{n}{m}\rfloor m\in D$. Then $n=0+n=(n-\lfloor \frac{n}{m}\rfloor m)+ \lfloor\frac{n}{m}\rfloor m$, where $0,n-\lfloor\frac{n}{m}\rfloor m \in C$ and $n$, $\lfloor \frac{n}{m}\rfloor m\in D$, so $r_{C,D}(n)\ge 2$, a contradiction.\vskip1mm

Subcase 2.2. $\lfloor\frac{n}{m}\rfloor m\not\in D$. It follows from $r_{C,D}(\lfloor \frac{n}{m}\rfloor m)=1$ that $\lfloor \frac{n}{m}\rfloor m=c+d$, where $d<\lfloor \frac{n}{m}\rfloor m$ and $m|d$.\vskip1mm

Subsubcase 2.2.1. $d>0$. Hence $c<\lfloor\frac{n}{m}\rfloor m$ and $m|c$. So $c+n-\lfloor \frac{n}{m}\rfloor m \in C$. Then $n=0+n=(c+n-\lfloor \frac{n}{m}\rfloor m)+d$, where $0,c+n- \lfloor \frac{n}{m}\rfloor m\in C$ and $n,d\in D$, so $r_{C,D}(n)\ge 2$, a contradiction.\vskip1mm

Subcase 2.2.2. $d=0$. Then $\lfloor\frac{n}{m}\rfloor m\in C$. Hence $\lfloor\frac{n}{m}\rfloor m, \lfloor\frac{n}{m}\rfloor m+1,\cdots,n-1\in C$. It follows that \begin{eqnarray*}n+m-1=(m-1)+n=(n-1)+m,\hskip3mm \mbox{where}\hskip3mm m-1,n-1\in C\hskip2mm \mbox{and}\hskip2mm n,m\in D,\end{eqnarray*} so $r_{C,D}(n-m-1)\ge 2$, a contradiction.\vskip1mm

\noindent{\emph{Case 3}}. $n\not\in C, n\not\in D$. Then $\lfloor\frac{n}{m}\rfloor m$, $\lfloor \frac{n}{m}\rfloor m+1,\cdots,n-1\in C$. If $n=c+d$, then $0<d\le\lfloor \frac{n}{m}\rfloor m$, so $d\ge m$ and $c<n$. Then $\lfloor\frac{c}{m}\rfloor m \in C$. It follows that \begin{eqnarray*}\lfloor \frac{n}{m}\rfloor m =\lfloor\frac{n}{m}\rfloor m+0=\lfloor \frac{c}{m}\rfloor m+d, \hskip3mm \mbox{where}\hskip3mm \lfloor \frac{n}{m}\rfloor m, \lfloor \frac{c}{m}\rfloor m\in C\hskip2mm \mbox{and}\hskip2mm d,0\in D,\end{eqnarray*} so $r_{C,D}( \lfloor \frac{n}{m}\rfloor m )\ge 2$, a contradiction.\vskip2mm

By $C=\{0,1,\cdots,m-1\}+mE$ and $D=mF$, we know that $C+D=\{0,1,\cdots,m-1\}+m(E+F)$. It follows from $r_{C,D}(n)=1$ for every nonnegative integer $n$ that $r_{E,F}(n)=1$ for every nonnegative integer $n$. This completes the proof of Lemma \ref{lemma1}.
\end{proof}

We now return to the proof of Theorem \ref{thm1}.\vskip2mm

\emph{Necessity.} By $r_{A,B}(0)=1$ we know that $0\in A\cap B$. We may assume that $1\in A$.
Now we will take induction on positive integer $k$ to prove that:\vskip2mm

\noindent\textbf{Proposition.} There exist a sequence $\{m_1,m_2,\cdots, m_{2k}\}$ of positive integers no less than two and sets $C_k,D_k\subset \mathbb{N}$, $0,1\in C_k$, $0\in D_k$, $r_{C_k,D_k}(n)=1$ for every nonnegative integer $n$, such that $A=A_k+m_1m_2\dots m_{2k}C_k$ and $B=B_k+m_1m_2\dots m_{2k}D_k$, where $A_k$ and $B_k$ are defined as follows (choose $m_0=1$, the same in the remaining proof):
\begin{eqnarray}\label{05102}
&&A_k=\{\epsilon_0+\epsilon_2m_1m_2+\cdots+\epsilon_{2k-2}m_1\cdots m_{2k-2}, \epsilon_{2i}=0,1,\cdots,m_{2i+1}-1\},\nonumber\\
&&B_k=\{\epsilon_1m_1+\epsilon_3m_1m_2m_3+\cdots+\epsilon_{2k-1}m_1\cdots m_{2k-1}, \epsilon_{2i-1}=0,1,\cdots,m_{2i}-1\}.
\end{eqnarray}

Firstly we apply Lemma \ref{lemma1} to the set $A$ and $B$, then we get that there exist an integer $m_1\ge 2$ and sets $C_1',D_1'\subseteq\mathbb{N}$, $0,1\in D_1'$, $0\in C_1'$, $r_{C_1',D_1'}(n)=1$ for every nonnegative integer $n$ and $B=m_1D_1'$, $A=\{0,1,\dots m_1-1\}+m_1C_1'$. Again, we apply Lemma \ref{lemma1} to the set $C_1'$ and $D_1'$, then we get that there exist an integer $m_2\ge 2$ and sets $C_1,D_1\subseteq\mathbb{N}$, $0,1\in C_1$, $0\in D_1$, $r_{C_1,D_1}(n)=1$ for every nonnegative integer $n$ and $C_1'=m_2C_1$, $D_1'=\{0,1,\dots m_2-1 \}+m_2D_1$. Hence $A=\{ 0,1,\dots m_1-1\}+m_1m_2C_1$, $B=\{0,m_1,\cdots,(m_2-1)m_1\}+m_1m_2D_1$. Namely, $A=A_1+m_1m_2C_1$ and $B=B_1+m_1m_2D_1$, $0,1\in C_1$, $0\in D_1$, $r_{C_1,D_1}(n)=1$ for every nonnegative integer $n$. Proposition holds for $k=1$.\vskip2mm

Assume that Proposition holds for some positive integer $k$, now we consider $k+1$.
We apply Lemma \ref{lemma1} to the set $C_k$ and $D_k$, then we get that there exist an integer $m_{2k+1}\ge 2$ and sets $C_{k+1}',D_{k+1}'\subseteq\mathbb{N}$, $0,1\in D_{k+1}'$, $0\in C_{k+1}'$, $r_{C_{k+1}',D_{k+1}'}(n)=1$ for every nonnegative integer $n$ and $D_k=m_{2k+1}D_{k+1}'$, $C_k=\{0,1,\dots m_{2k+1}-1\}+m_{2k+1}C_{k+1}'$. Again, we apply Lemma \ref{lemma1} to the set $C_{k+1}'$ and $D_{k+1}'$, then we get that there exist an integer $m_{2k+2}\ge 2$ and sets $C_{k+1},D_{k+1}\subseteq\mathbb{N}$, $0,1\in C_{k+1}$, $0\in D_{k+1}$, $r_{C_{k+1},D_{k+1}}(n)=1$ for every nonnegative integer $n$ and $C_{k+1}'=m_{2k+2}C_{k+1}$, $D_{k+1}'=\{0,1,\cdots,m_{2k+2}-1\}+m_{2k+2}D_{k+1}$. It follows from the induction hypothesis that
\begin{eqnarray*}
A&=&A_k+m_1m_2\dots m_{2k}C_k\\
&=&A_k+m_1m_2\dots m_{2k}(\{0,1,\dots m_{2k+1}-1\}+m_{2k+1}m_{2k+2}C_{k+1})\\
&=&A_{k+1}+m_1m_2\dots m_{2k+2}C_{k+1},\\
B&=&B_k+m_1m_2\dots m_{2k}D_k\\
&=&B_k+m_1m_2\dots m_{2k}(\{0,1,\cdots,m_{2k+2}-1\}m_{2k+1}+m_{2k+1}m_{2k+2}D_{k+1})\\
&=&B_{k+1}+m_1m_2\dots m_{2k+2}D_{k+1},\end{eqnarray*}
where $0,1\in C_{k+1}$, $0\in D_{k+1}$, $r_{C_{k+1},D_{k+1}}(n)=1$ for every nonnegative integer $n$. Thus, Proposition holds for $k+1$. As $k\rightarrow\infty$, \eqref{05101} holds. Then \emph{Necessity} follows.\vskip2mm

This completes the proof of Theorem \ref{thm1}. \vskip3mm

\noindent\textbf{Proof of Theorem \ref{thm2}.} If $r_{A,B}(n)=1$ for every nonnegative integer $n$, then by Theorem \ref{thm1} we know that sets $A,B$ are with the form \eqref{05101} (or $A,B$ interchanged). We only need to consider sets $A,B$ with the form \eqref{05101}, the other case is completely similar. We will firstly prove that for every positive integer $k$, we have \begin{eqnarray}\label{05133} A_k-B_k&=&[-(m_2-1)m_1-(m_4-1)m_1m_2m_3-\cdots-(m_{2k}-1)m_1m_2\cdots m_{2k-1},\nonumber\\
&&(m_1-1)+(m_3-1)m_1m_2+\cdots+(m_{2k-1}-1)m_1m_2\cdots m_{2k-2}],\end{eqnarray}
where the sets $A_k,B_k$ are defined in \eqref{05102}. Write the right side in \eqref{05133} as $I_k$. \vskip2mm

Clearly, $A_1=\{0,1,\cdots,m_1-1\}$, $B_1=\{0,m_1,\cdots,(m_2-1)m_1\}$. Hence $A_1-B_1=I_1$.
\eqref{05133} holds for $k=1$. Assume that \eqref{05133} holds for some positive integer $k$, namely, $A_k-B_k=I_k$. It follows that \begin{eqnarray*}A_{k+1}-B_{k+1}=\bigcup_{i=0}^{m_{2k+1}-1}\bigcup_{j=0}^{m_{2k+2}-1}((im_1m_2\cdots m_{2k}-jm_1m_2\dots m_{2k+1})+I_k)=I_{k+1}.\end{eqnarray*}
Thus, \eqref{05133} holds for $k+1$.\vskip2mm

By \eqref{05133} and $|A_k-B_k|\le m_1m_2\dots m_{2k}=|I_k|$, we know that $d_{A_k,B_k}(n)=1$ for every $n\in I_k$. Noting that \begin{eqnarray*}A=\bigcup_{k=1}^{\infty} A_k,\hskip4mm
B=\bigcup_{k=1}^{\infty} B_k,\end{eqnarray*} as $k\rightarrow\infty$, we could deduce from
\eqref{05133} that $d_{A,B}(n)=1$ for every integer $n$. This completes the proof of Theorem \ref{thm2}.\vskip3mm

Before the proof of Theorem \ref{thm3}, we first introduce the following nice lemma from \cite[Lemma 2.1]{Ma}.
\vskip2mm

\begin{lemma}\label{lemma2} (\cite{Ma})
Let $m_1,m_2,\cdots$ be arbitrary integers no less than two. Then the sets $A$ and $B$ with the form \eqref{05101} are infinite additive complements such that
\begin{eqnarray}\label{05104}
\limsup_{x\rightarrow\infty}\frac{A(x)B(x)}{x}=\limsup_{x\rightarrow\infty}\frac{2}{1+D_k},
\end{eqnarray}
where
\begin{eqnarray*}
D_k=\frac{1}{m_k}-\frac{1}{m_km_{k-1}}+\frac{1}{m_km_{k-1}m_{k-2}}-\cdots
+(-1)^{k-1}\frac{1}{m_km_{k-1}\cdots m_1}.\end{eqnarray*}
\end{lemma}\vskip2mm

\noindent\textbf{Proof of Theorem \ref{thm3}.}
If $r_{A,B}(n)=1$ for every nonnegative integer $n$, then
\begin{eqnarray}\label{05111}\liminf_{x\rightarrow\infty}\frac{A(x)B(x)}{x}\ge1.\end{eqnarray}
We may assume that $1\in A$. It follows from Theorem \ref{thm1} that $A,B$ are with the form \eqref{05101}. For $x_k=m_1m_2\cdots m_{2k}-1$, we have $A(x_k)=m_1m_3\cdots m_{2k-1}, B(x_k)=m_2m_4\cdots m_{2k}$ and hence $A(x_k)B(x_k)-x_k=1$. It follows from \eqref{05111} that
\begin{eqnarray*}\liminf_{x\rightarrow\infty}\frac{A(x)B(x)}{x}=1.\end{eqnarray*}\vskip2mm

If $m_k\ge 3$ for infinitely many $k$, then $D_k<\frac{1}{m_k}\le \frac{1}{3}$ for infinitely many $k$. If $m_k=2$ for all large $k$, then $\lim_{k\rightarrow\infty} D_k=\frac{1}{3}$. In both cases, we have $\liminf_{k\rightarrow\infty}D_k\le \frac{1}{3}$. Obviously, $\liminf_{k\rightarrow\infty}D_k\ge 0$. It follows from Lemma \ref{lemma2} that
\begin{eqnarray*}\frac{3}{2}\le \limsup_{x\rightarrow\infty}\frac{A(x)B(x)}{x}\le 2.
\end{eqnarray*}\vskip2mm

Furthermore, the constants $\frac{3}{2}$ and $2$ are best possible by taking $m_k=2$ for every $k$ and $m_k=k$ for every $k$, respectively.\vskip2mm

This completes the proof of Theorem \ref{thm3}.\vskip3mm

\noindent\textbf{Proof of Theorem \ref{thm4}.} By \cite[Theorem 2]{Sandor}, there exist two infinite sets $A_1=\{a_1<a_2<\cdots\}$ and $B_1=\{b_1<b_2<\cdots\}$ of nonnegative integers such that $r_{A,B}(n)\ge 1$ for all large $n$ and \begin{equation*}\lim_{x\to +\infty}\frac{A_1(x)B_1(x)}{x}=1.\end{equation*}
For every positive integer $n$, let \begin{eqnarray*}T_n=\max\{a_{n^4},b_{n^4}\}\end{eqnarray*}
and \begin{eqnarray*}C_n=\bigcup_{k=0}^{n} \{T_n+(2^{2k-1}+2-2^{k-1})-(2^k-2)k\},\hskip4mm D_n=\bigcup_{k=0}^{n} \{T_n+(2^{2k-1}+2-2^{k-1})-(2^k-1)k\}.\end{eqnarray*}
Take \begin{eqnarray*}A=A_1\bigcup \bigcup_{n=1}^{\infty} C_n\hskip3mm  \mbox{and}\hskip3mm B=B_1\bigcup\bigcup_{n=1}^{\infty} D_n.\end{eqnarray*}
It follows from the construction of $A,B$ that for any given nonnegative integer $c$, we have $d_{A,B}(n)=c$ for infinitely many positive integers $n$. Furthermore,
\begin{eqnarray*}
A_1(x)\le A(x)\le A_1(x)+\sqrt{A_1(x)}\hskip3mm \mbox{and}\hskip3mm B_1(x)\le B(x)\le B_1(x)+\sqrt{B_1(x)}.\end{eqnarray*}
Thus, $A,B$ are infinite additive complements and \eqref{1} holds. This completes the proof of Theorem \ref{thm4}.


\begin{thebibliography}{99}

\bibitem{Danzer} L. Danzer, \"Uber eine Frage von G. Hanani aus der additiven Zahlentheorie, J. Reine Angew. Math. 214/215 (1964), 392-394.

\bibitem{Sandor} S.Z. Kiss, C. S\'andor, On a problem of Chen and Fang related to infinite additive complements, Acta Arith. 200 (2021), 213-220.

\bibitem{Ma} F.Y. Ma, A note on additive complements, arXiv:2205.04128.

\end{thebibliography}
\end{document}